\documentclass[12pt]{amsart}
\usepackage[latin1]{inputenc}
\usepackage{indentfirst}
\usepackage{amsfonts}
\usepackage{graphicx}
\usepackage{amsmath}
\usepackage{amssymb}
\usepackage{latexsym}
\usepackage{amscd}
\usepackage[all]{xy}

\newcommand{\F}{\mathbb {F}}

\newtheorem{theorem}{Theorem}[section]
\newtheorem{definition}[theorem]{Definition}
\newtheorem{lemma}[theorem]{Lemma}
\newtheorem{corollary}[theorem]{Corollary}
\newtheorem{proposition}[theorem]{Proposition}
\newtheorem{remark}[theorem]{Remark}

\def \ord {\rm{ord}\,}

\begin{document}

\title[On existence of some special pair of primitive elements over $\mathbb{F}_{2^k}$]{On existence of some special pair of primitive elements over finite fields}

\author[C.\ Carvalho, J.P.G.\ Sousa, V.G.L.\ Neumann and 
G.\ Tizziotti]{C\'icero Carvalho, Jo\~ao Paulo Guardieiro Sousa, Victor 
G.L. Neumann and Guilherme Tizziotti}

\maketitle

\begin{abstract}
In this paper we generalize the results of Sharma, Awasthi and Gupta (see 
\cite{SAG}). We work over a field of any characteristic with $q = p^k$ elements 
and we give a sufficient condition for the existence of a primitive element 
$\alpha \in \mathbb{F}_{p^k}$ such that $f(\alpha)$ is also primitive in 
$\mathbb{F}_{p^k}$, where $f(x) \in \mathbb{F}_{p^k}(x)$ is a quotient of 
polynomials with some restrictions. We explicitly determine the values of $k$ 
for which such a pair exists for $p=2,3,5$ and $7$.
\end{abstract}

\section{Introduction}
Let $\mathbb{F}_q$ denote a finite field with $q$ elements, where $q = p^k$ and
$p$ is a prime. In this paper we present  results on the problem of finding
pairs $(\alpha, \beta)$ of elements in $\mathbb{F}_q$ which are related in some
way, and  are both
primitive, i.e.\ both generate the multiplicative group $\mathbb{F}_q^*$. The
origins of this problem may be traced back to a question asked by A.\ Brauer to
his former student E.\ Vegh about the existence of ``pairs of consecutive
primitive roots'' in $\mathbb{F}_p$, meaning, in the above notation, that
$\beta = \alpha + 1$. Vegh proved that if $p > 3$ and $\varphi(p - 1)/(p - 1) >
1/3$, where $\varphi$ is the Euler totient function, then there exists such a
pair (see \cite{vegh1}). Three years later he proved, among other results, that
if $p \equiv 1 \pmod{4}$ then the condition $\varphi(p - 1)/(p - 1) >
1/4$ is enough to guarantee the existence of such pair (see \cite{vegh2}).
In 1985, in a series of two papers (see \cite{Cohen1985a, Cohen1985b}) Cohen
extended the result, proving that if $q > 3$ and $q \not\equiv 7 \pmod{12}$
then there exists such a pair of primitive elements. From there the question
whether there exist pairs of primitive elements $(\alpha, \beta)$ where $\beta$
is a polynomial in $\alpha$ was a natural one. It was considered, especially
for polynomials of degree one and two, by several authors (see e.g.\ the survey
\cite{Cohen1993}) and in 2015 Cohen, Oliveira e Silva and Trudgian proved,
among other results, that
for every $q > 61$ there exists a pair $(\alpha, \beta)$ of primitive elements
with $\beta = a \alpha + b$, where $a, b \in \mathbb{F}_q$. For polynomials of
degree two, the most recent result is by Booker, Cohen and Sutherland (see
\cite{Cohen2019}), and it
states that for for every $q > 211$ there exists a pair $(\alpha, \beta)$ of
primitive elements
with $\beta = a \alpha^2  + b \alpha + c$, with $b^2 - 4 a c \neq 0$.

Of course, one can also consider  pairs $(\alpha, \beta)$ of primitive elements
such that $\beta$ is a rational function of $\alpha$, and recently several
authors have studied this problem. One of the simplest case
is trivial: $\beta = 1/\alpha$ is primitive if and only if $\alpha$ is
primitive. In 2012 Wang, Cao and Feng (see \cite{WCF})  proved, among other
results, that if $q = 2^{s n}$ then there exists a pair $(\alpha, \beta)$ of
primitive elements
such that $\beta = \alpha + 1/\alpha$ provided that $n$ is odd, $n \geq 13$ and
$s > 4$.
Two years later their result was generalized by Cohen (see \cite{Cohen}), who
proved, also assuming that $q$ is a power of 2, that if $q \geq 8$ then there
exists a pair $(\alpha, \beta)$ of
primitive elements
such that $\beta = \alpha + 1/\alpha$.  Also in 2014 Kapetanakis presented
necessary conditions for the existence of a pair
$(\alpha, \beta)$ of
primitive elements
such that $\beta =( a \alpha + b)/(c \alpha + d)$, with $a, b, c, d \in
\mathbb{F}_q$. In 2017 Anju and Sharma found necessary conditions for
existence of a pair
$(\alpha, \beta)$ of
primitive elements
such that $\beta =( a \alpha^2 + b \alpha + c)/(d \alpha + e)$, with $a, b, c,
d, e \in
\mathbb{F}_q$ and $q = 2^k$. In 2018 Sharma, Awasthi and Gupta, again working
over a finite field of characteristic two, determined conditions for the
existence of  a pair
$(\alpha, \beta)$ of
primitive elements
such that $\beta =( a \alpha^2 + b \alpha + c)/(d \alpha^2 + e \alpha + f)$,
with $a, b, c,
d, e, f \in
\mathbb{F}_{2^k}$.

In the present paper we generalize the results of Sharma, Awasthi and Gupta,
finding conditions that guarantee existence of  a pair
$(\alpha, \beta)$ of
primitive elements
such that $\beta$ is a quotient of polynomial expressions in $\alpha$. We work
over a field of any characteristic with $q = p^k$ elements. Our main result is
Theorem \ref{main teo} which is used to determine the values of $k$ for which
such a pair exists, although sometimes we also need Lemma \ref{lema3.5} to
determine this values. The next section contains the definition of the set 
$\Gamma_p(m_1,m_2)$, and basic concepts and results that will be used 
throughout the work. In Section \ref{main section}, we present the general 
results about the set $\Gamma_p(m_1,m_2)$. Finally, in Section \ref{examples} 
we 
determine the sets $\Gamma_p(3,2)$ for $p=2,3,5$ and $7$.

\section{Preliminaries}

Throughout the paper, $p$ is a prime, $k$ is a positive integer and
$\mathbb{F}_q$ will denote a finite field with $q=p^k$ elements, and we denote
by $\mathbb{N}$ the set of positive integers.

An element $\alpha\in\mathbb{F}_q$ is called
\textit{primitive} if $\alpha$ is a generator of the multiplicative group
$\mathbb{F}_q^*$, or equivalently, if the multiplicative order of $\alpha$ is $q-1$. A
pair $(\alpha,\beta) \in \mathbb{F}_q^2$ is a \textit{primitive pair} in
$\mathbb{F}_q$ if $\alpha$ and $\beta$ are primitive elements. It is clear that
$(\alpha,\beta) \in \mathbb{F}_q^2$ is a primitive pair if and only if
$(\alpha,\beta^{-1}) \in \mathbb{F}_q^2$ is a primitive pair.

%Given a positive integer $n$ and a polynomial $g \in \mathbb{F}_q [x]$, two coprimes polynomials $f_1$ and $f_2$ are called $(g^n,q)$-\textit{coprimes} if $g^n$ divides the product $f_1 f_2$ but $g^{n+1}$ not divides $f_1f_2$, and $\mbox{gcd}(n,q-1)=1$.

The following concepts will play a crucial role in what follows.

\begin{definition}
i) Let  $f_1, f_2 \in \mathbb{F}_q [x]$, we define  $\Lambda_q(f_1,f_2)$ as the set of pairs $(n,g) \in \mathbb{N} \times (\mathbb{F}_q [x]
\setminus \{x\})$ such that $\gcd(n,q-1) = 1$, $g$ is monic, irreducible, $g^n \mid f_1 f_2$ and  $g^{n+1} \nmid f_1 f_2$. \\
ii) Let $m_1,m_2 \in \mathbb{N}$, we define  $\Upsilon_{q} (m_1,m_2)$ as the set of rational functions $\frac{f_1}{f_2} \in \mathbb{F}_q(x)$
such that
$\deg (f_1) \leq m_1$,  $\deg (f_2) \leq m_2$,   $\gcd(f_1,f_2)=1$ and $\Lambda_{q}(f_1,f_2) \neq \emptyset$.\\
iii) Let $m_1,m_2 \in \mathbb{N}$, we define $\Gamma_p(m_1,m_2)$ as the set of positive integers $k$ such that
$\mathbb{F}_{p^k}$  contains an element $\alpha$ with $(\alpha, f(\alpha))$  a primitive pair  for all $f\in \Upsilon_{p^k} (m_1,m_2)$.
\end{definition}

In \cite{mersenne} and \cite{SAG}, the authors studied sets similar to 
$\Gamma_2(2,1)$ and $\Gamma_2(2,2)$, respectively. We comment on this 
similarity at the beginning of Section \ref{examples}.\\

The next result give us some properties of the sets above. 

\begin{proposition}\label{consequencias}
Let $p$ be a prime and let $k,\ell_1,\ell_2, m_1,m_2$ be positive integers. Then,
\begin{enumerate}
\item $\Upsilon_{p^k} (m_1,m_2) \subsetneqq \Upsilon_{p^k} (m_1 + \ell_1,m_2 + \ell_2)$;
\item $\Gamma_p(m_1,m_2) = \Gamma_p(m_2,m_1)$;
\item $\Gamma_p(m_1+\ell_1 ,m_2 + \ell_2) \subset \Gamma_p(m_1,m_2)$.
\end{enumerate}
\end{proposition}

\begin{proof}
\begin{enumerate}
\item The inclusion is obvious, and for the inequality it suffices to take $f_1$ to be an irreducible polynmial of degree $m_1+\ell_1$ and $f_2 = 1$.
\item This is a direct consequence of the fact that $(\alpha,\beta) \in \mathbb{F}_q^2$ is a primitive pair if and only if $(\alpha,\beta^{-1}) \in \mathbb{F}_q^2$ is a primitive pair.
\item It is a direct consequence of (1).
\end{enumerate}
\end{proof}

In what follows we will need the following result, which is a special case of  \cite[Theorem 5.5]{Fu}.
\begin{lemma}  \label{lema cota}
Let $h(x) \in \mathbb{F}_{q}(x)$ be a rational function. Write $h(x) =
\prod_{j=1}^{r} h_{j}(x)^{n_j}$, where $h_j(x) \in \mathbb{F}_{q}[x]$ are
irreducible polynomials and $n_j$ are nonzero integers. Let $\chi$ be a
multiplicative character of $\mathbb{F}_{q}$. Suppose that the rational
function $h(x)$ is not of the form $g(x)^{ord(\chi)}$
in $\mathbb{F}(x)$, where $\mathbb{F}$ is an algebraic closure of $\mathbb{F}_q$. Then we have
$$
\displaystyle \left|
\sum_{\substack{\alpha \in \mathbb{F}_q  \\ h(\alpha)\neq 0 , h(\alpha)\neq
\infty}} \chi(h(\alpha)) \right|
\leq
\left(\sum_{j=1}^{r} \deg(h_j)-1 \right) q^{\frac{1}{2}}.
$$
\end{lemma}
%
%\medskip
%
%A (multiplicative) character of $\mathbb{F}_q^*$ of order $d$ will be denoted
%by $\chi_d$.
%
%\begin{remark} \label{obs ordem}
%By \cite[Th. 1.15]{LN} we know that there are $\phi(d)$ characters of
%$\mathbb{F}_q^*$ of order $d$, where $\phi$ is the Euler's function.
%\end{remark}
%
%For more details about characters of finite fields see \cite[Chapter 5]{LN}.
%

\begin{definition}
Let  $s$ by a divisor of $q-1$,
an element $\alpha\in\mathbb{F}_q^*$ is called \textit{$s$-free} if,
for any $d \in \mathbb{N}$ such that $d\, |\, s$ and $d \neq 1$,
 there is no $\beta\in\mathbb{F}_q$ satisfying $\beta^d = \alpha$.
\end{definition}

Thus, if $\alpha \in\mathbb{F}_q^*$ is $s$-free then  $\alpha$ cannot be
a $d$-th power of an element, where $d\neq1$ and $d|s$.
Next, we present some easy remarks that we will use in what follows.

\begin{remark} \label{obs s_free}
Let $\alpha\in\mathbb{F}_q^*$, then:

\begin{enumerate}

\item $\alpha$ is primitive if, and only if, $\alpha$ is $(q-1)$-free;

\item if $\alpha$ is $s$-free for some integer $s$, then $\alpha$ is $e$-free, for any $e \, | \, s$;

\item if $\alpha$ is $s_1$-free and $s_2$-free, then $\alpha$ is $\mathrm{lcm}(s_1,s_2)$-free;

\item let $p_1, \ldots, p_n$ be primes, then $\alpha$ is $(p_1 .\cdots. p_n)$-free if and only if $\alpha$ is  $(p_{1}^{\alpha_1} .\cdots. p_{n}^{\alpha_n})$-free, where $\alpha_i > 0$ for all $i$.

\end{enumerate}

\end{remark}

From \cite{CH2} we get that the characteristic function $\rho_s$ of the set of $s$-free elements is given by
\begin{equation} \label{funcao caracteristica}
\alpha\mapsto\theta(s)\sum_{d|s}\frac{\mu(d)}{\phi(d)}\sum_{\mathrm{ord}(\chi) = d}\chi(\alpha),
\end{equation}
where $\theta(s):=\frac{\phi(s)}{s}$, $\mu$ is the Moebius's function, and $\mathrm{ord}(\chi)$ denotes the order of the multiplicative character $\chi$.

\section{Main results} \label{main section}

Let $m_1$ and $m_2$ be positive integers, our aim is to determine for
which $k$ there exists an  $\alpha \in \mathbb{F}_{p^k}$ such that  $(\alpha, f(\alpha))$ is a primitive pair for all $f \in \Upsilon_{p^k}(m_1,m_2)$.
For this we will need the following concept.

\begin{definition}
Let $q = p^k$ and let  $l_1$ and $l_2$ be divisors of $q - 1$.
Given $f\in \Upsilon_{q}(m_1,m_2)$ we will denote by $N_f(l_1,l_2)$ the number of pairs $(\alpha,f(\alpha))$ such that $\alpha \in \mathbb{F}_q$ is
$l_1$-free and $f(\alpha)$ is $l_2$-free.
\end{definition}

Thus, in order to have $k \in \Gamma_p(m_1,m_2)$, we must have $N_f(q-1,q-1)>0$ for all $f\in \Upsilon_{q}(m_1,m_2)$.

For an integer $\ell$, we denote by $\omega(\ell)$ and $W(\ell)$ the number of prime divisors of $\ell$ and the
number of square-free divisors of $\ell$, respectively, clearly one has $W(\ell)$=$2^{\omega(\ell)}$.

\begin{theorem} \label{main teo}
	Let $f=\frac{f_1}{f_2}\in \Upsilon_{q}(m_1,m_2)$, with $q \geq 4$. If $q^{\frac{1}{2}}>(m_1+m_2)W(l_1)W(l_2)$, then $N_f(l_1,l_2)>0$.
\end{theorem}
\begin{proof}
Let $f=\frac{f_1}{f_2}\in \Upsilon_{q}(m_1,m_2)$ and let
$$
S_f:=\{ \beta \in \mathbb{F}_{q} \; \colon \; f_1(\beta) =0 \mbox{ or } f_2(\beta) =0 \} \cup \{0\}.
$$
 Using the characteristic functions $\rho_{l_1}$ and $\rho_{l_2}$ we
 have
 $$
 N_f(l_1,l_2)=\sum_{\alpha\in\mathbb{F}_{q} \backslash
 S_f}\rho_{l_1}(\alpha)\rho_{l_2}(f(\alpha)),
 $$
 and from the formulas for these functions we get
	\begin{equation}\label{eq1}
	N_f(l_1,l_2)=\theta(l_1)\theta(l_2)\sum_{d_1|l_1, d_2|l_2}\frac{\mu(d_1)\mu(d_2)}{\phi(d_1)\phi(d_2)}
	\sum_{\substack{\ord(\chi_1)=d_1 \\ \ord(\chi_2)=d_2}}
	\tilde{\chi}(\chi_{1},\chi_{2}),
	\end{equation}
	where
$\tilde{\chi}(\chi_{1},\chi_{2}):=\sum\limits_{\alpha\in\mathbb{F}_{q}\backslash S_f}\chi_{1}(\alpha)\chi_{2}(f(\alpha))$.

Let $\chi_1$ and $\chi_2$ be multiplicative characters of orders $d_1$ and
$d_2$, respectively, where $d_1 \, |\, l_1$ and $d_2 \, | \, l_2$.
Let $i \in \{1, 2\}$,
it is
well-known (see e.g. \cite[Thm.\ 5.8]{LN}) that there exists a character
$\chi$  of order $q-1$ and and integer
$n_i \in\{0,1,...,q-2\}$ such that $\chi_{i}(\alpha)=\chi(\alpha^{n_i})$
for all $\alpha \in \mathbb{F}_{q}^*$, and observe that
$n_i=0$ if and only if $d_i=1$.
%
%Fixemos um par de caracteres $\chi_1, \chi_2$ de ordem $d_1,d_2$
%respectivamente, onde $d_1,d_2$ s\~ao inteiros positivos que satisfazem
%$d_1 \mid l_1$ e $d_2 \mid l_2$.	
%From \cite[Thm.\ 5.8]{LN}, we know that there are $n_1, n_2\in\{0,1,...,q-2\}$
%such that $\chi_{i}(\alpha)=\chi(\alpha^{n_i})$ for $i=1,2$
%	and for all $\alpha \in \mathbb{F}_{q}^*$, where $\chi$ is a fixed
%character of order $q-1$. Observe that $n_i=0$ if and only if $d_i=1$.
Hence,
\begin{eqnarray}
	\tilde{\chi}(\chi_{1},\chi_{2}) & = &\sum_{\alpha\in\mathbb{F}_{q}\setminus S_f}\chi(\alpha^{n_1}f_1(\alpha)^{n_2}f_2(\alpha)^{-n_2}) \nonumber \\
	& = & \sum_{\alpha\in\mathbb{F}_{q}\setminus S_f}\chi(h(\alpha)), \nonumber
\end{eqnarray}
where $h(x)=x^{n_1}f_1(x)^{n_2}f_2(x)^{-n_2}$.
To find a bound for $N_f(l_1,l_2)$ we will
bound the above summation
according to the values of $d_1$ and $d_2 $. We consider three cases.

%
%
%$n_2=0$ (and $d_2=1$). This implies that
%$x^{n_1}g_2(x)^{q-1}=g_1^{q-1}(x)$, which is possible only if $n_1=0$ (and so
%$d_1=1$).
%
%In a similar way if $t(x)^n$ is a factor of $f_1(x)$, we conclude that $n_1=0$
%(and $d_1=1$), which implies
%$n_2=0$ (and $d_2=1$).
%
%
%To give a bound for the number $N_f(l_1,l_2)$ we will separate the sums given
%in
%\eqref{eq1} in five cases depending on the possible values of $d_1$ and $d_2$.
%
%Calculemos $\tilde{\chi}(\chi_{1},\chi_{2})$ em fun\c{c}\~ao dos valores de
%$d_1$ e $d_2$.

\begin{enumerate}
\item[(i)] We first consider the case where $d_1=1$ and $d_2=1$.
Then $n_1=0$, $n_2=0$ and $h=1$, so that
$$
\tilde{\chi}(\chi_{1},\chi_{2}) = \sharp (\mathbb{F}_q \setminus S_f) \geq q - (m_1+m_2+1).
$$

\item[(ii)] Now we deal with the case where $d_1\neq 1$ and $d_2=1$.
Then  $n_2=0$ and $h(\alpha)=\alpha^{n_1}$.
Observe that
$\displaystyle \sum_{\alpha \in \mathbb{F}_q^*}\chi(\alpha^{n_1})= \sum_{\alpha
\in \mathbb{F}_q^*}\chi_1(\alpha)= 0$, so that
\begin{eqnarray}
\left| \tilde{\chi}(\chi_{1},\chi_{2}) \right| & =&
\left| \sum_{\alpha \in \mathbb{F}_q^*}\chi(\alpha^{n_1})
-
\sum_{\alpha\in\mathbb{F}_{q}\setminus S_f}\chi(\alpha^{n_1})
\right| =
\left| \sum_{\alpha \in S_f\setminus\{0\}}\chi(\alpha^{n_1}) \right|
\nonumber \\
& \leq &(m_1+m_2) < (m_1+m_2)q^{\frac{1}{2}}.
\nonumber
\end{eqnarray}

\item[(iii)] Lastly we consider the case where $d_2 \neq 1$, so that $n_2 \neq
0$. To bound
$\tilde{\chi}(\chi_{1},\chi_{2})$ we want to use Lemma \ref{lema cota}, and we
start by showing that indeed we can use it. So we assume by means of absurd
that $h(x)=\left(
\frac{g_1(x)}{g_2(x)} \right)^{q-1}$
for some
$g_1(x),g_2(x)\in\mathbb{F}[x]$, with $\textrm{deg}(g_1)=r_1$,
$\textrm{deg}(g_2)=r_2$, and $\mbox{gcd}(g_1,g_2)=1$, then
\begin{equation*}
x^{n_1}f_1(x)^{n_2}g_2(x)^{q-1}=f_2(x)^{n_2}g_1^{q-1}(x).
\end{equation*}
Since $\frac{f_1(x)}{f_2(x)} \in \Upsilon_{q}
(m_1,m_2)$, there exists an
irreducible monic polynomial $t(x) \in \mathbb{F}_{q}[x]$, $t(x)\neq x$ and a
positive integer $n$
with $\mbox{gdc}(n,q-1)=1$ such that
$t(x)^n$ appears in the factorization of either $f_1(x)$ or $f_2(x)$. Let's
suppose that $t(x)^n$ appears in the factorization of $f_2(x)$, and let
$\tilde{t}(x)$ be an irreducible factor of $t(x)$ in $\mathbb{F}[x]$. Clearly
$\tilde{t}(x)$ has degree one, $\tilde{t}(x) \neq x$ and since $\mathbb{F}_q$
is a
perfect field we
know that $\tilde{t}(x)$ appears with multiplicity one in the factorization of
$t(x)$ in $\mathbb{F}[x]$.
Since $f_1(x)$ and $f_2(x)$ are coprime in $\mathbb{F}_q[x]$ they are also
coprime in  $\mathbb{F}[x]$ so
$\tilde{t}(x)^{n n_2}$ appears in the factorization of $g_2(x)^{q-1}$. From
this one may conclude that
$q-1 \, | \, n n_2$, and from $\gcd(n,q-1)=1$ we get $q-1 \, | \, n_2$, a
contradiction. So we must have that
$t(x)^n$ appears in the factorization of $f_1(x)$, and reasoning as above again
we
conclude that $q-1\, | \, n_2$, which is impossible. Thus, if $d_2 \neq 1$ then
$n_2 \neq 0$ and
we get that $h(x)$ is not of the form $g(x)^{q - 1}$
in $\mathbb{F}(x)$.

Let $T_h$ be the set of $\beta \in \mathbb{F}_q$ such that $h(\beta) = 0$ or
$h(\beta)$ is not defined. If $0 \in T_h$ then $T_h = S_f$ and from Lemma
\ref{lema cota} we have
$$
\left| \tilde{\chi}(\chi_{1},\chi_{2}) \right|
=
\left| \sum_{\alpha\in\F_{q}\setminus S_f}\chi(h(\alpha))
\right| =
\left| \sum_{\alpha\in\F_{q}\setminus T_h}\chi(h(\alpha))
\right|
\leq (m_1 + m_2) q^{\frac{1}{2}}.
$$
If $0 \notin T_h$ then %$n_1 = 0$ and
$$
\left| \tilde{\chi}(\chi_{1},\chi_{2}) \right|
=
\left| \sum_{\alpha\in\F_{q}\setminus S_f}\chi(h(\alpha))
\right| =
\left| \sum_{\alpha\in\F_{q}\setminus T_h}\chi(h(\alpha)) - \chi(h(0))
\right|
\leq (m_1 + m_2 - 1) q^{\frac{1}{2}} + 1.
$$
so anyway we get $| \tilde{\chi}(\chi_{1},\chi_{2}) | \leq (m_1 + m_2)
q^{\frac{1}{2}}$.
\end{enumerate}

Now we use the above estimates to bound $ N_f(l_1,l_2)$. In the
second inequality below we use that
 there are $\phi(d_1)$ characters of order $d_1$ and $\phi(d_2)$
characters of order $d_2$, so we have $\phi(d_1)\phi(d_2)$ pairs
of such characters. From \eqref{eq1} and what we did above, we get
\[
\begin{split}
N_f(l_1,l_2) &\geq
\theta(l_1)\theta(l_2)
\left( q-(m_1+m_2+1) -
\sum_{\substack{d_1|l_1, d_2|l_2 \\ (d_1,d_2) \neq
(1,1)}}\frac{|\mu(d_1)| |\mu(d_2)|}{\phi(d_1)\phi(d_2)}
	\sum_{\substack{\ord(\chi_1)=d_1 \\ \ord(\chi_2)=d_2}}
	\left| \tilde{\chi}(\chi_{1},\chi_{2}) \right| \right) \\
&\geq
\theta(l_1)\theta(l_2)
\left( q-(m_1+m_2+1) -
\sum_{\substack{d_1|l_1, d_2|l_2 \\ (d_1,d_2) \neq
(1,1)}}\frac{|\mu(d_1)| |\mu(d_2)|}{\phi(d_1)\phi(d_2)}
	\sum_{\substack{\ord(\chi_1)=d_1 \\ \ord(\chi_2)=d_2}}
	(m_1+m_2)q^{\frac{1}{2}} \right) \\
&\geq
\theta(l_1)\theta(l_2)
\left(q-(m_1+m_2+1)-(m_1+m_2)q^{\frac{1}{2}}
\sum_{\substack{d_1|l_1, d_2|l_2 \\ (d_1,d_2) \neq (1,1)}}
|\mu(d_1)||\mu(d_2)|
\right)
\end{split}
\]
The last sum is equal to the number of pairs $(d_1,d_2)\neq(1,1)$ such that
$d_1\, |\, l_1$ and $d_2\, |\, l_2$ with $d_1$ and $d_2$ square-free, since the
Moebius function returns $0$ in the other cases, so
\begin{equation} \label{eq W}
\displaystyle \sum_{\substack{d_1|l_1, d_2|l_2 \\ (d_1,d_2) \neq (1,1)  }}|\mu(d_1)||\mu(d_2)| = W(l_1)W(l_2)-1.
\end{equation}
Therefore, we conclude that
\begin{equation}\label{eq desigualdade N}
\displaystyle N_f(l_1,l_2) > \theta(l_1)\theta(l_2)
\left(
q-(m_1+m_2+1)-(m_1+m_2)q^{\frac{1}{2}}(W(l_1)W(l_2)-1)
\right).
\end{equation}
Thus, if $q \geq (m_1+m_2+1)+(m_1+m_2)q^{\frac{1}{2}}(W(l_1)W(l_2)-1)$, then
$N_f(l_1,l_2)>0$. Writing
\[
\begin{split}
(m_1+m_2+1) + &(m_1+m_2)q^{\frac{1}{2}}(W(l_1)W(l_2)-1) \\ &=
(m_1+m_2)q^{\frac{1}{2}}W(l_1)W(l_2) - q^{\frac{1}{2}}\left( (m_1+m_2) -
q^{\frac{-1}{2}} (m_1+m_2+1)\right)
\end{split}
\]
and noting that $(m_1+m_2) -
q^{\frac{-1}{2}} (m_1+m_2+1) \geq 0$ if and only if
$(m_1 + m_2) (q^{\frac{1}{2}} - 1) \geq 1$ we observe that, for a fixed $q$,
the least value on the left occurs when $m_1 + m_2 = 1$, and then we must have
$q \geq 4$. Thus, if $q \geq 4$ and
$q^{\frac{1}{2}}\geq (m_1+m_2)W(l_1)W(l_2)$ we have that $N_f(l_1,l_2)>0$.
\end{proof}

\begin{corollary}\label{mainresult}
If $q \geq 4$ and
$q^{\frac{1}{2}} \geq (m_1 + m_2) W(q-1)^2$ then $k\in  \Gamma_p(m_1,m_2)$.
\end{corollary}

The next two results, and their proofs, are similar, respectively, to \cite[Lemma 3.3]{CH3} and
\cite[Theorem 3.8]{Cohen}.

%Next we obtain a similar result of the sieving Lemma 3.4 of \cite{mersenne} or Lemma 3.3 of \cite{CH3}, and the proof follows the reasoning presented in \cite{CH3}.

\begin{lemma}\label{lemmasieve}
	Let $\ell$ be a divisor of $q-1$ and let $\{p_1,...,p_r\}$ be the set of all primes which divide $q-1$ but do not divide $\ell$. Then
\begin{equation}\label{sieve}
N_f(q-1,q-1)\geq\sum_{i=1}^{r}N_f(p_i \ell,\ell)+\sum_{i=1}^{r}N_f(\ell,p_i\ell)-(2r-1)N_f(\ell,\ell).
\end{equation}
\end{lemma}
\begin{proof}
The left side of \eqref{sieve} counts every $\alpha \in \F_q^n$ for which both  $\alpha$ and $f(\alpha)$ are primitive.
Observe that if $\alpha$ and $f(\alpha)$ are primitive then they are also  $p_i \ell$-free and $\ell$-free, so they are counted $r+r - (2r-1)=1$
times on the right side of \eqref{sieve}.
For any other $\alpha \in \F_q^n$, from  Remark \ref{obs s_free}, we have that either $\alpha$ or $f(\alpha)$ is not $p_i \ell$-free for some
$i \in \{ 1,\ldots , r\}$, so the right side of \eqref{sieve} is at most zero.	
\end{proof}

%The next result give us more information about the set $\Gamma_p(m_1,m_2)$. The proof is essentially the same as \cite[Theorem 3.8]{Cohen}.

\begin{lemma}\label{lema3.5}
Let $\ell$ be a divisor of $q-1$ and let $\{p_1,...,p_r\}$ be the set of all primes which divide $q-1$ but do not divide $\ell$.
	Suppose that $\delta=1-2\sum_{i=1}^{r}\frac{1}{p_i}>0$ and let  $\Delta=\frac{2r-1}{\delta}+2$. If
	$q^{\frac{1}{2}} \geq (m_1 + m_2) W(\ell)^2  \Delta$,
	then $k\in  \Gamma_p(m_1,m_2)$.
\end{lemma}
\begin{proof}
Recall that $\theta(p_i) = 1 - \dfrac{1}{p_i}$, for all $i = 1, \ldots, r$,
then we may rewrite the right side of \eqref{sieve} obtaining

\begin{equation}\label{eq N}
\begin{split}
N_f(q-1,q-1)\geq &\sum_{i=1}^{r}(N_f(p_i \ell,\ell) - \theta(p_i)N_f(\ell,\ell)) \\& +\sum_{i=1}^{r}(N_f(\ell,p_i\ell) - \theta(p_i)N_f(\ell,\ell)) + \delta N_f(\ell,\ell).
\end{split}
\end{equation}
	
Since $\theta(p_i \ell) = \theta(p_i) \theta(\ell)$, for all $i = 1, \ldots, r$, from Equation (\ref{eq1}) we get
$$
N_f(p_i \ell,\ell)=\theta(p_i)\theta(\ell)^2
\sum_{d_1|p_i \ell, d_2|\ell}\frac{\mu(d_1)\mu(d_2)}{\phi(d_1)\phi(d_2)}
\sum_{\substack{\ord(\chi_1)=d_1 \\ \ord(\chi_2)=d_2}}
\tilde{\chi}(\chi_1,\chi_2),
$$
	for all $i$.
	
We split the set of $d_1$'s which divide $p_i \ell$ into two sets: the first one contains those which do not have $p_i$ as a factor, while the seconde
one contains those which are a mutiple of $p_i$. This will split the first summation into two sums, and we get
%
%
%Separating the sum
%$\displaystyle \sum_{d_1|p_i \ell, d_2|\ell}$
%in two parts, the first one given by elements $d_1 | \ell$ and the second one given by
%$p_i | d_1$ and $d_1 | p_i \ell$, we have
%$\displaystyle \sum_{d_1|p_i \ell, d_2|\ell}\frac{\mu(d_1)\mu(d_2)}{\phi(d_1)\phi(d_2)}\sum_{\chi_{d_1},\chi_{d_2}}\tilde{\chi}(\chi_{d_1},\chi_{d_2}) = \sum_{d_1| \ell, %d_2|\ell}\frac{\mu(d_1)\mu(d_2)}{\phi(d_1)\phi(d_2)}\sum_{\chi_{d_1},\chi_{d_2}}\tilde{\chi}(\chi_{d_1},\chi_{d_2}) + \sum_{p_i |d_1|p_i \ell, %d_2|\ell}\frac{\mu(d_1)\mu(d_2)}{\phi(d_1)\phi(d_2)}\sum_{\chi_{d_1},\chi_{d_2}}\tilde{\chi}(\chi_{d_1},\chi_{d_2})$,
%we have
\begin{equation*}
\begin{split}
N_f(p_i \ell,\ell)&=\theta(p_i)\theta(\ell)^2
\sum_{d_1| \ell, d_2|\ell}\frac{\mu(d_1)\mu(d_2)}{\phi(d_1)\phi(d_2)}
\sum_{\substack{\ord(\chi_1)=d_1 \\ \ord(\chi_2)=d_2}}
\tilde{\chi}(\chi_1,\chi_2) \\&
+ \theta(p_i)\theta(\ell)^2
\sum_{p_i | d_1, d_1|p_i \ell, d_2|\ell}\frac{\mu(d_1)\mu(d_2)}{\phi(d_1)\phi(d_2)}
\sum_{\substack{\ord(\chi_1)=d_1 \\ \ord(\chi_2)=d_2}}
\tilde{\chi}(\chi_1,\chi_2)
\end{split}
\end{equation*}
and from the expression for $N_f(\ell,\ell)$ (see equation \eqref{eq1}) we get

$$
N_f(p_i \ell,\ell) - \theta(p_i) N_f(\ell, \ell) = \theta(p_i) \theta(\ell)^2
\sum_{\substack{ p_i |d_1, d_1|p_i\ell \\ d_2|\ell}}
\frac{\mu(d_1)\mu(d_2)}{\phi(d_1)\phi(d_2)}
\sum_{\substack{\ord(\chi_1)=d_1 \\ \ord(\chi_2)=d_2}}
\tilde{\chi}(\chi_1,\chi_2).
$$
	
From (ii) and (iii) in the proof of Theorem \ref{main teo} and from
$$
\displaystyle \sum_{\substack{p_i|d_1, d_1| p_i l \\ d_2|l  }}|\mu(d_1)||\mu(d_2)| = W(l)^2
$$
we conclude that
$$
\left| N_f(p_i \ell,\ell) - \theta(p_i) N_f(\ell, \ell) \right| \leq
(m_1 + m_2) \theta(p_i) \theta(\ell)^2  W(\ell)^2 q^{\frac{1}{2}}.
$$
In a similar way, we can conclude that
$$
\left| N_f(\ell,p_i \ell) - \theta(p_i) N_f(\ell, \ell) \right| \leq
(m_1 + m_2) \theta(p_i) \theta(\ell)^2  W(\ell)^2 q^{\frac{1}{2}},
$$
for all $i=1,\ldots,r$.
	
Replacing those results on inequality (\ref{eq N}) we get
$$
\displaystyle N_f(q-1, q-1) \geq \delta N_f(\ell, \ell) - 2(m_1 + m_2) \theta(\ell)^2  W(\ell)^2 q^{\frac{1}{2}} \sum_{i=1}^{r} \theta(p_i).
$$

    From $\displaystyle \sum_{i=1}^{r}\theta(p_i) = \dfrac{\delta}{2} \left( \dfrac{2r-1}{\delta} + 1 \right) = \dfrac{\delta}{2} (\Delta - 1)$, where $\Delta=\frac{2r-1}{\delta}+2$, we get

    $$
	\displaystyle N_f(q-1, q-1) \geq \delta N_f(\ell, \ell) - (m_1 + m_2) \delta(\Delta - 1) \theta(\ell)^2   W(\ell)^2 q^{\frac{1}{2}}.
	$$

    Since $N_f(\ell,\ell) \geq \theta(\ell)^2
\left( q-(m_1+m_2+1)-(m_1+m_2)q^{\frac{1}{2}}(W(\ell)^2-1)\right)$ (see equation (\ref{eq desigualdade N})), we have that
 $$
	\displaystyle N_f(q-1, q-1) \geq \delta \theta(\ell)^2  \left( q - (m_1 + m_2 + 1)  + (m_1+m_2)q^{\frac{1}{2}} - (m_1 + m_2) \Delta W(\ell)^2q^{\frac{1}{2}} \right)  .
$$
	
	 From the hypothesis we have $\delta >0$, and since $(m_1 + m_2)q^{\frac{1}{2}} - (m_1 + m_2 + 1) > 0$, we conclude that if $q^{\frac{1}{2}} \geq (m_1 + m_2) W(\ell)^2  \Delta$, then $N_f(q-1, q-1) > 0$ and therefore $k\in  \Gamma_p(m_1,m_2)$.
\end{proof}

For any $p$ we have the following result.

\begin{proposition}\label{noGamma}
If $\phi(q-1) \leq m_1 + m_2+1$, then $k \notin \Gamma_p(m_1,m_2)$.
\end{proposition}
\begin{proof}
Let $ \{ \alpha_1, \ldots , \alpha_{\phi(q-1)}\}$ the set of all primitive elements of $\mathbb{F}_{q}$. Note that, if $\phi(q-1) \leq m_1 + m_2+1$, we may choose polynomials $f_1(x)$ and $f_2(x)$
of degrees $m_1$ and $m_2$, respectively, such that $f_1(\alpha_j) f_2(\alpha_j)=0$, for all $j=1,\ldots,\phi(q-1) - 1$, and $f_1(\alpha_{\phi(q-1)}),f_2(\alpha_{\phi(q-1)}) \neq 0$, and $f(x) = \frac{f_1(x)}{f_2(x)} \in \Upsilon_{q} (m_1,m_2)$. So, $(\alpha_j, f(\alpha_j))$ is not a primitive pair for all $j=1,\ldots,\phi(q-1)-1$. Taking $\beta = \frac{1}{f(\alpha_{\phi(q-1)})}$ we have that $h(x)=\beta f(x) \in \Upsilon_{q} (m_1,m_2)$ and $(\alpha_{\phi(q-1)}, h(\alpha_{\phi(q-1)})) = (\alpha_{\phi(q-1)},1)$ is not  a primitive pair either. Thus $k \notin \Gamma_p(m_1,m_2)$.
\end{proof}

When $p=2$ we can use the following results to gather  information
on the set $\Gamma_2(m_1,m_2)$. The proposition below is a generalization of
\cite[Theorem 3.7]{mersenne}.

\begin{proposition}\label{prop2}
Let $m= max \{ m_1 , m_2\}$ and $k > 1$. If $2^k-1$ is a prime number and
$2^k-2 > m_1
+ m_2 + m $ then $k\in \Gamma_2(m_1,m_2)$.
\end{proposition}
\begin{proof}
	Assume that  $2^k-1$ is a prime, then we have a total of $2^k - 2$
	primitive elements in $\mathbb{F}_{2^k}$, since  every $ \alpha \in
	\mathbb{F}_{2^k}^* \setminus \{1\}$ is primitive.
	Let $f(x)=\frac{f_1(x)}{f_2(x)}\in
	\Upsilon_{2^k}(m_1,m_2)$ and assume that $2^k-2 > m_1
	+ m_2 + m $.  Since $f_1(x)$ and $f_2(x)$ have at most $m_1$
	and $m_2$ roots,
	respectively, and $f_1(x)-f_2(x)$ has at most $m$ roots, with $m= max \{
	m_1 , m_2\}$ then there exists a primitive element $\alpha \in
		\mathbb{F}_{2^k}$ such that $f_1(\alpha)\neq0$, $f_2(\alpha)\neq0$
			and $f_1(\alpha) \neq f_2(\alpha)$, so that $f(\alpha)$ is
			primitive.
\end{proof}

\begin{proposition}\label{prop3}
	Let $q=2^k$ and
	$m= max \{ m_1 , m_2\}$. If
	$$
	\phi(q-1) + \frac{1}{m}\phi(q-1)>q,
	$$
	then $k\in \Gamma_2(m_1,m_2)$.
\end{proposition}
\begin{proof}
	Let $f(x)=\frac{f_1(x)}{f_2(x)}\in \Upsilon_{q}(m_1,m_2)$ and define
	$$
	A_f = \{\alpha \in \F_q^* \mid \alpha \text{ is primitive and } f_2(\alpha)
	\neq 0\}.
	$$
	The rational function $f$ defines a function from $A_f$ to $\F_q$ given by $\alpha \longmapsto f(\alpha)$.
	Let $B_f = \mathop{\rm{Im}} f$. For a given $\beta \in B_f$ there are at
	most $m$
	elements $\alpha \in A_f$
	such that $f(\alpha)=\beta$, since $\alpha$ must be a zero of the
	polynomial $f_1(x)-\beta f_2(x)$.
	Clearly we have  $|A_f| \geq \phi(q-1)-m$ which implies $|B_f| \geq
	\frac{\phi(q-1)-m}{m}$. Thus, if
	$\frac{\phi(q-1)-m}{m}+\phi(q-1)>q-1$ then at least one element of $B_f$,
	say $f(\alpha)$, is
	primitive
	for $\alpha \in A_f$, so that $(\alpha,f(\alpha))$ is a primitive pair and
	$k\in \Gamma_p(m_1,m_2)$.
\end{proof}

The above result cannot be extended to the case where
$p$ is an odd prime since in this case we have that $q - 1$ is even and
$\phi(q-1) + \frac{1}{m}\phi(q-1)<q$ for any $m \geq 1$.

%
%A partir do Corollary \ref{mainresult} podemos deduzir resultados assimpt\'oticos.
%
%Note that knowing $W(q-1)$ gets hard while $q$ grows, so a bound for this
%value is very useful. In \cite[Lemma 6.2]{Cohen}, the authors prove that if
%$m$
%is an
%odd positive integer then $W(m)< 6.46 m^{\frac{1}{5}}$.
%In what follows we  use a similar reasoning to  present a variavel bound for
%$W(q-1)$.

\begin{proposition}\label{cota-t}
Let $p$ be a prime, $q=p^k$,
$m_1,m_2$ positive integers,  $t >4$ a positive real number and
$$
A_t=\prod_{\substack{s \textrm{ prime }\\s < 2^t } }\frac{2}{\sqrt[t]{s}}.
$$
If
\begin{equation*} %\label{cotaAt}
q \geq
\left(
(m_1+m_2) \cdot A_t^2
\right)^{\frac{2t}{t-4}},
\end{equation*}
then $k \in \Gamma_p(m_1,m_2)$.
\end{proposition}
\begin{proof}
We note initially that
\[
\begin{split}
&q \geq
\left(
(m_1+m_2) \cdot A_t^2
\right)^{\frac{2t}{t-4}} \iff
q^{\frac{t-4}{2t}} \geq
\left(
(m_1+m_2) \cdot A_t^2
\right) \\ &\iff
q^{\frac{1}{2}} \geq
\left(
(m_1+m_2) \cdot A_t^2 \right) q^{\frac{2}{t}}
\iff
q^{\frac{1}{2}} \geq
(m_1+m_2) \cdot (A_t \cdot q^{\frac{1}{t}})^2.
\end{split}
\]

Let $q-1=p_1^{\alpha_1}\cdots p_l^{\alpha_l}$, where $p_1, \ldots, p_l$ are
positive distinct primes, then $W(q-1)=2^{l}$. If $p > 2^t$ then
$\frac{2}{\sqrt[t]{p}}<1$, and we get
$$\frac{W(q-1)}{q^{\frac{1}{t}}}
<\frac{W(q-1)}{(q-1)^{\frac{1}{t}}}
=\frac{2^{l}}{\sqrt[t]{p_1^{\alpha_1}}\cdots\sqrt[t]{p_l^{\alpha_l}}}
\leq\prod_{i=1}^{l} \frac{2}{\sqrt[t]{p_i}}
\leq
\prod_{\substack{2 \leq p_i \leq 2^t \\ 1 \leq i \leq l}}\frac{2}{\sqrt[t]{p_i}}\leq A_t.$$
This implies that
$A_t \cdot q^{\frac{1}{t}} > W(q-1)$, hence
$q^{\frac{1}{2}}
> (m_1+m_2) W(q-1)^2$.
From Corollary \ref{mainresult} we get $k \in \Gamma_p(m_1,m_2)$.
\end{proof}

\begin{remark}
To illustrate the use of the above result, we take, for example, $t=6$ in that Proposition and we get, for any prime $p$, that:
\begin{enumerate}
\item If $q=p^k \geq 5.6 \times 10^{21}$ then $k \in \Gamma_p(2,1)$.
\item If $q=p^k \geq 3.2 \times 10^{22}$ then $k \in \Gamma_p(2,2)$.
\item If $q=p^k \geq 1.2 \times 10^{23}$ then $k \in \Gamma_p(3,2)$.
\end{enumerate}
\end{remark}

\section{Working Examples} \label{examples}
In \cite{mersenne} the authors define a set of matrices $\mathcal{M}_q$ and for each
$A \in \mathcal{M}_q$ they associate a rational function $\lambda_A(x) \in 
\F_q(x)$ which is a quotient of a polynomial of degree at most 2 by another of 
degree at most one. In that paper they work over a finite field with $2^k$ 
elements, and they want to investigate the existence of primitive elements 
$\alpha$ such that $\lambda_A(\alpha)$ is also primitive. 
It is easy to check that if $f \in \Upsilon_{q} (2,1)$ then there exists
$A \in \mathcal{M}_q$ such that $f=\lambda_A$. Conversely, if $A \in \mathcal{M}_q$ then
$\lambda_A \in \Upsilon_{q} (2,1)$ unless $\lambda_A(x)=x , x^2$ or $\beta x^{-1}$ for some
$\beta \in \F_q$. Observe that for $\lambda_A(x)=x , x^2$ or $\beta x^{-1}$ there always exists an
$\alpha \in \F_q$ such that $(\alpha, \lambda_A(\alpha))$ is a primitive pair. 
They also define a set $\mathcal{B}$ of powers of $2$ which satisfy a similar 
condition 
as the elements of
$\Gamma_2(2,1)$, and we get that  
$q=2^k \in \mathcal{B}$ if and only if $k \in \Gamma_2(2,1)$, and from 
\cite[Theorem 1.1]{mersenne}, we get
$\Gamma_2(2,1) = \mathbb{N} \setminus \{1, 2,  4\}$.

In \cite{SAG} the authors define a set of matrices $N_{2 \times 3} (\F_q)$ in a 
similar way as
it was done in \cite{mersenne} with the difference
that in this case the associated rational function $\lambda_A(x)$ is quotient
of polynomials of degree two. They also want to study 
the existence of primitive elements 
$\alpha$ such that $\lambda_A(\alpha)$ is also primitive. 
It is easy to
see that if $f \in \Upsilon_{q} (2,2)$ and $f \notin \Upsilon_{q} (2,1) \cup \Upsilon_{q} (1,2)$, then
there exists a matrix $A \in N_{2 \times 3} (\F_q)$ such that $f=\lambda_A(x)$.
Observe also that $\Upsilon_{q} (2,1) \cup \Upsilon_{q} (1,2) \subset 
\Upsilon_{q} (2,2)$,
so $\Gamma_2(2,2) \subset \mathbb{N} \setminus \{1, 2,  4\}$. Conversely, if
$A \in N_{2 \times 3} (\F_q)$ then $\lambda_A \in \Upsilon_{q} (2,2)$, and from
\cite[Theorem 1.5]{SAG} we get $\mathbb{N} \setminus \{1, 2,4,6,8,9,10,12\} 
\subset \Gamma_2(2,2)$. Also in
\cite{SAG} the authors conclude that $k\in\Gamma_2(2,2)$ for all positive 
integer $k$, with exception of $k\in\{1, 2, 4, 6, 8, 9, 10, 12\}$, a result 
that may  be  recovered from  
Proposition \ref{consequencias}(3). They also prove that $1,2,4 
\notin \Gamma_2(2,2)$
and conjecture that $6, 8, 9, 10, 12 \in \Gamma_2(2,2)$, from Proposition 
\ref{prop2}
we get $9 \in \Gamma_2(2,2)$.

Now, we will study the set $\Gamma_p(3,2)$ for $p=2,3,5,7$, and we start with the case where $p = 2$.
%and we note that from Proposition \ref{noGamma} it follows that $\{1,2,3\} \cap \Gamma_2(3,2) = \emptyset$, and
%by Proposition \ref{prop2} we have, for example, that $\{5, 7, 13, 17, 19\} \subset \Gamma_2(3,2)$

\begin{proposition}\label{proposition1}
$\mathbb{N} \setminus \{1, 2, 3, 4, 6, 8, 10, 12\}  \subset \Gamma_2(3,2)$
and $\{1,2,3,4\} \cap \Gamma_2(3,2) = \emptyset$.
\end{proposition}
\begin{proof}
From Proposition \ref{cota-t} and using $t=6$ we get that
$k \in \Gamma_2(3,2)$ for all $k \geq 77$.
From Corollary \ref{mainresult}, for $k \geq 4$, if $2^{\frac{k}{2}} > 5 W(2^k-1)^2$ then
$k\in  \Gamma_2(3,2)$.
For $k < 77$, we used SAGEMATH (see \cite{SAGE}) to factor $2^k-1$ and conclude that $2^{\frac{k}{2}}>5W(2^k-1)^2$ for $k=13,17,19$ and for $k\geq 21$
with the exception of
$k=24,28,36$.
	
Using Lemma \ref{lema3.5}, with $q = 2^k$,  we get that for
$$
k \in \{11,14,15,16,18,20,24,28,36\}
$$
we have $k \in \Gamma_2(3,2)$.
Table \ref{dadosp2_3_2} subsumes these results.
\begin{table}[h]
		\centering
		\begin{tabular}{ccc|cccc}
			$k$ & $\ell$ & $\{p_1,p_2,\ldots ,p_r\}$ &$\quad$ &
			$k$ & $\ell$ & $\{p_1,p_2,\ldots ,p_r\}$ \\
			\hline
			$11$ & $1$ & $\{ 23,89\}$  &&
                      $20$ & $3 \cdot 5^2$ & $\{ 11,31,41\}$ \\
			$14$ & $3$ & $\{ 43,127\}$  & &
			          $24$ & $3^2 \cdot 5$  & $\{ 7,13,17,241\}$ \\
			$15$ & $1$ & $\{ 7,31,151\}$  &&
       			      $28$ & $3 \cdot 5$  & $\{ 29,43,113,127\}$  \\
			$16$ & $3$ & $\{ 5,17,257\}$  &&
					  $36$ & $3^3 \cdot 5$  & $\{ 7,13,19,37,73,109\}$ \\
			$18$ & $3^3$ & $\{ 7,19,73\}$  & &  & &
		\end{tabular}
\caption{Data that satisfies Lemma \ref{lema3.5} (hence $k \in \Gamma_2(3,2)$).}
\label{dadosp2_3_2}
\end{table}

From Proposition \ref{noGamma} we get that $1,2,3 \notin \Gamma_2(3,2)$.
In \cite{mersenne} it is shown that
for any primitive element $\beta \in \F_{16}$ we have that $f(\beta)$ is not primitive, where
$f=\frac{\alpha x+1}{x+\alpha} \in \F_{16}[x]$ and $\alpha \in \F_{16}$ is
a fixed primitive element of $\F_{16}$,
so
$4 \notin \Gamma_2(1,1)$ and in particular $4 \notin \Gamma_2(3,2)$.

Finally, from
Proposition \ref{prop2} we get that $5, 7 \in \Gamma_2(3,2)$ and from Proposition
\ref{prop3} we get also $9 \in \Gamma_2(3,2)$.
\end{proof}

%\begin{proof}
%	By Theorem \ref{main teo}, for $k \geq 4$, if $2^{\frac{k}{2}}>6W(2^k-1)^2$ then $N_f(2^k-1,2^k-1)>0$.
%	From
%	\cite[Lemma 6.2]{Cohen} we get  that if $m$ is an
%	odd positive integer then $W(m)< 6.46 m^{\frac{1}{5}}$, thus
%	 $$W(2^k-1)<6.46(2^k-1)^{\frac{1}{5}}<6.46(2^k)^\frac{1}{5}.$$ So, if $2^\frac{k}{2}>6(6.46)^22^{\frac{2k}{5}}$, then $q=2^k\in\mathcal{M}$. Now, $2^\frac{k}{2}>6(6.46)^22^{\frac{2k}{5}}$ is equivalent to $2^{\frac{k}{10}} > 250.3896$. Therefore we can conclude that $k \in \Gamma_2(3,3)$ for all $k \geq 80$.
%
%	For $k < 80$, we used SAGEMATH (see \cite{SAGE}) to factor $2^k-1$ and conclude that $2^{\frac{k}{2}}>6W(2^k-1)^2$ for $k\geq21$, with the exception of $k=24, 28, 36$.
%
%
%Using Lemma \ref{lema3.5} and SAGEMATH (\cite{SAGE}) to calculate $W(\ell)^2$ and $\Delta$ we get
%$$\{11, 14, 15, 16, 18, 20, 24, 28, 36\} \subset \Gamma_2(3,3).$$
%
%Finally, combining the results obtained in this section, we have the following result.
%\end{proof}

Now we proceed to study the sets $\Gamma_3(3,2)$, $\Gamma_5(3,2)$ and $\Gamma_7(3,2)$.

\begin{theorem}
For the sets $\Gamma_3(3,2)$, $\Gamma_5(3,2)$ and $\Gamma_7(3,2)$ the following results hold:
\begin{enumerate}
\item[i)]
$\{1,2\} \cap \Gamma_3(3,2) = \emptyset$ and $\{9, 10, 11\} \cup \{ k \in \mathbb{N} \colon k \geq 13\} \subset \Gamma_3(3,2)$;

\item[ii)]
$1 \notin \Gamma_5(3,2)$ and $\{ k \in \mathbb{N} \colon k \geq 7\} \subset \Gamma_5(3,2)$;

\item[iii)]
$1 \notin \Gamma_7(3,2)$ and $\{ k \in \mathbb{N} \colon k \geq 7\} \subset \Gamma_7(3,2)$.
\end{enumerate}

%
%	$1,2 \notin \Gamma_3(3,2)$,
%	$k \in \Gamma_3(3,2)$ para $k=9,10,11$ e $k \geq 13$;
%	$1 \notin \Gamma_5(3,2)$ e
%	$k \in \Gamma_5(3,2)$ para $k \geq 7$;
%	$1 \notin \Gamma_7(3,2)$ e
%	$k \in \Gamma_7(3,2)$ para $k \geq 7$.
\end{theorem}
\begin{proof}
From Proposition \ref{cota-t} and using $t=6$ we get that
$k \in \Gamma_3(3,2)$ for all $k \geq 49$,
$k \in \Gamma_5(3,2)$ for all $k \geq 34$ and
$k \in \Gamma_7(3,2)$ for all $k \geq 28$.
	
From Corollary \ref{mainresult}, for $k \geq 4$, if $p^{\frac{k}{2}}>5W(p^k-1)^2$ then
$k\in  \Gamma_p(3,2)$.
We used SAGEMATH (see \cite{SAGE}) to factor $p^k-1$
and conclude that $p^{\frac{k}{2}}>5W(p^k-1)^2$ for the following values of $k$.
\begin{table}[h]
	\centering
	\begin{tabular}{ccc}
		$p$ && $k \in \Gamma_p(3,2)$ \\
		\hline
		$3$ && $11 \leq k \leq 48$, except $k=12,18$   \\
		$5$ && $7 \leq k \leq 33$, except $k=8,10,12$  \\
		$7$ && $8 \leq k \leq 28$
	\end{tabular}
\caption{$k\in  \Gamma_p(3,2)$ using $p^{\frac{k}{2}}>5W(p^k-1)^2$}
\label{dadosp_3_2}
\end{table}

The next tables summarize the results obtained using Lemma \ref{lema3.5} for
$p=3,5,7$.
\begin{table}[h]
	\centering
	\begin{tabular}{ccc}
		$k$ & $\ell$ & $\{p_1,p_2,\ldots ,p_r\}$ \\
		\hline
		$9$   & $2$ & $\{ 13,757\}$  \\
		$10$ & $2$  & $\{ 11,61\}$  \\
		$18$ & $2$  & $\{ 7,13,19,37,757\}$
	\end{tabular}
\caption{For $p=3$ we get $9,10,18 \in \Gamma_3(3,2)$}
\end{table}

\begin{table}[h]
	\centering
	\begin{tabular}{ccc}
		$k$ & $\ell$ & $\{p_1,p_2,\ldots ,p_r\}$ \\
		\hline
		$7$   & $2$ & $\{ 19531\}$  \\
		$8$   & $2$ & $\{ 3,13,313\}$  \\
		$10$ & $2$  & $\{ 3,11,71,521\}$  \\
		$12$ & $2$  & $\{7,13,31,601\}$
	\end{tabular}
\caption{For $p=5$ we get $7,8,10,12 \in \Gamma_5(3,2)$}
\end{table}

\begin{table}[h]
	\centering
	\begin{tabular}{ccc}
		$k$ & $\ell$ & $\{p_1,p_2,\ldots ,p_r\}$ \\
		\hline
		$7$   & $2\cdot 3$ & $\{ 29,43 \}$  \\
	\end{tabular}
	\caption{For $p=7$ we get $7 \in \Gamma_7(3,2)$}
\end{table}

Finally from Proposition \ref{noGamma} we get that
$1,2 \notin \Gamma_3(3,2)$, $1 \notin \Gamma_5(3,2)$ and $1 \notin \Gamma_7(3,2)$.
\end{proof}

\end{document}